\documentclass[reqno]{amsart}

\usepackage{hyperref}
\usepackage{mathtools, comment}

\newtheorem{lemma}{Lemma}[section]
\newtheorem{theorem}[lemma]{Theorem}
\newtheorem{proposition}[lemma]{Proposition}
\newtheorem{corollary}[lemma]{Corollary}

\theoremstyle{definition}

\newtheorem{question}{Question}

\usepackage{amssymb}
\renewcommand{\leq}{\leqslant}
\renewcommand{\geq}{\geqslant}

\renewcommand{\emptyset}{\varnothing}

\def\calC{\mathcal{C}}

\def\F{\mathbf F}

\newcommand{\SU}{\mathrm{SU}}
\newcommand{\PGL}{\mathrm{PGL}}

\newcommand{\AGL}{\mathrm{AGL}}
\newcommand{\PSL}{\mathrm{PSL}}

\newcommand{\PGammaL}{\mathrm{P\Gamma L}}

\newcommand\floor[1]{\left\lfloor{#1}\right\rfloor}
\newcommand\ceil[1]{\left\lceil{#1}\right\rceil}

\usepackage{todonotes}

\begin{document}

\title{Large minimal invariable generating sets in the finite symmetric groups}

\author{Daniele Garzoni}
\address{Daniele Garzoni, Dipartimento di Matematica ``Tullio Levi-Civita'', Universit\`a degli Studi di Padova, Padova, Italy}
\email{daniele.garzoni@phd.unipd.it}

\author{Nick Gill}
\address{Nick Gill, Department of Mathematics, University of South Wales, Treforest CF37 1DL, U.K.}
\email{nick.gill@southwales.ac.uk}

\begin{abstract}
   For a finite group $G$, let $m_I(G)$ denote the largest possible cardinality of a minimal invariable generating set of $G$. We prove an upper and a lower bound for $m_I(S_n)$, which show in particular that $m_I(S_n)$ is asymptotic to $n/2$ as $n\rightarrow \infty$.
\end{abstract}

\maketitle

\section{Introduction}


Let $G$ be a group, and let $I=\{\mathcal{C}_1,\dots, \mathcal{C}_k\}$ be a set of conjugacy classes of $G$. We say that $I$ \emph{invariably generates} $G$ if $\langle x_1,\dots, x_k \rangle=G$ for all $x_1\in\mathcal{C}_1,\dots, x_k\in\mathcal{C}_k$.  The set $I$ is a \emph{minimal invariable generating set} (or \emph{MIG-set} for short) if $I$ invariably generates $G$ and no proper subset of $I$ invariably generates $G$. We let $m_I(G)$ denote the largest possible cardinality of a MIG-set for the group $G$. 

We state our main theorem.

\begin{theorem}\label{t: main}
Let $n\geq 2$ be an integer, and let $G=S_n$ be the symmetric group on $n$ letters. Then
\[
\frac{n}{2}-\log n < m_I(G) < \frac n2 +\Delta(n)+O\left(\frac{\log n}{\log \log n}\right),
\]
where $\Delta(n)$ is the number of divisors of $n$.
\end{theorem}

(Logarithms are in base $2$.) It is well known that $\Delta(n)= n^{o(1)}$, therefore Theorem \ref{t: main} implies that $m_I(S_n)$ is asymptotic to $n/2$ as $n\rightarrow \infty$.

The parameter $m_I(G)$ is the ``invariable'' analogue of the parameter $m(G)$, which denotes the largest possible cardinality of a minimal generating set of $G$ (with analogous definition). See Subsection \ref{subsec: context_mIm} for more context.

In \cite{GarLuc}, it was asked whether $m_I(G)\leq m(G)$ holds for every finite group $G$, and it was proved that $m_I(G)=m(G)$ in case $G$ is soluble. In particular, $m_I(S_n)=m(S_n)$ for $n\leq 4$. 

In the upper bound in Theorem \ref{t: main}, we will in fact prove a more explicit estimate, which will have the following consequence.

\begin{corollary}\label{c: main}
If $G=S_n$ with $n\geq 5$, then $m_I(G)<m(G)$.
\end{corollary}

As we shall explain in the next subsection, this was known for large enough $n$.

\subsection{Methods of proof} For a finite group $G$, we denote by $k(G)$ the number of conjugacy classes of $G$. We begin with an elementary observation.

\begin{lemma}
\label{l: basic_obs_mI}
Suppose that $I=\{\mathcal{C}_1,\dots, \mathcal{C}_t\}$ is a set of conjugacy classes of a non-trivial finite group $G$. Then, $I$ is a MIG-set if and only if the following conditions are both satisfied:

\begin{itemize}
    \item[(a)] There exists a set of maximal subgroups $J=\{M_1,\dots, M_t\}$ of $G$ such that, for every $i\neq j$, $\calC_i \cap M_j\neq \emptyset$.
    \item[(b)] No proper subgroup of $G$ has non-empty intersection with $\mathcal{C}_i$ for all $i=1,\dots, t$.
\end{itemize}
\end{lemma}

In particular, with notation as in the previous lemma, we have the following two consequences:
\begin{itemize}
    \item[$\diamond$] For every $j=1, \ldots, t$, $M_j$ has non-empty intersection with at least $t-1$ non-trivial $G$-conjugacy classes, and therefore $k(M_j)\geq t$.
    \item[$\diamond$] For every $i\neq j$, $M_i$ and $M_j$ are not $G$-conjugate.
\end{itemize}

\noindent {\bf The lower bound.} Let us see how Lemma \ref{l: basic_obs_mI} can be exploited to give a lower bound for $m_I(S_n)$. Each conjugacy class $\mathcal{C}$ of $S_n$ corresponds to a particular partition, $X_\mathcal{C}$, of the integer $n$. On the other hand, if $M$ is an \emph{intransitive} subgroup of $S_n$, then $M$ is the stabilizer of some $i$-subset of $\{1,\dots, n\}$. 

We say that \emph{the integer $i$ is a partial sum of the partition $\mathfrak p=(a_1,\dots, a_t)$} if we can write $i=a_{j_1}+a_{j_2}+\cdots + a_{j_\ell}$ for some $1\leq j_1 < \cdots < j_\ell \leq t$. It is clear that the intersection $\mathcal{C}\cap M$ is non-empty if and only if $i$ is a partial sum of $X_\calC$. We will prove the following:

\begin{proposition}\label{p: partitions}
Let $n\geq 5$ be an integer. There is a set $X$ of partitions of $n$ with the following properties:
\begin{enumerate}
    \item There is no integer $1 \leq i \leq n/2$ which is a partial sum in $x$ for every $x \in X$;
    \item For every $x \in X$, there exists an integer $1 \leq i \leq n/2$ which is a partial sum in $y$ for every $y \in X \setminus \{x\}$;
    \item $|X|>\frac 12n - \log n$.
    \end{enumerate}
\end{proposition}

Proposition~\ref{p: partitions} is almost enough to yield the lower bound in Theorem~\ref{t: main} straight away. To complete the proof of that lower bound, we must take care of Lemma \ref{l: basic_obs_mI}(b) for proper transitive subgroups of $S_n$.

Our feeling is that the construction we give in our proof of Proposition~\ref{p: partitions} is pretty close to being as large a set $X$ as is possible.

\begin{question}
\label{q: large_partitions}
Is it true that the largest cardinality of a set $X$ of partitions of $n$ satisfying properties (1) and (2) of Proposition~\ref{p: partitions} is at most $\frac 12n-\log n+O(1)$?
\end{question}

Note that we certainly have $|X|\leq \frac 12n$.
\vspace{3pt}

\noindent {\bf The upper bound.} 
In view of Lemma \ref{l: basic_obs_mI} and of the considerations following it, it is clear that the upper bound in Theorem \ref{t: main} will be established once we prove the following result.\footnote{The idea we are using here is laid out explicitly in \cite{GarLuc}: For every maximal subgroup $M$ of $G$, denote by $M^*$ the set of $G$-conjugacy classes having non-empty intersection with $M$. Let 
\[
\mathcal M(G)=\{M^*\mid M \text { maximal subgroups of $G$}\}.
\]
We say that a subset $\{X_1,\dots,X_t\}$ of $\mathcal M(G)$ is \textit{independent} if, for every  $1\leq i\leq t$, the intersection $\cap_{j\neq i}X_j$ properly contains $\cap_j X_j$. We denote by $\iota(G)$ the largest cardinality of an independent subset of $\mathcal M(G)$. It is not hard to see, first, that $m_I(G)\leqslant\iota(G)$ (\cite[Lemma 4.2]{GarLuc}) and, second, that Proposition~\ref{p: upper} yields an upper bound for $\iota(S_n)$.}

\begin{proposition}\label{p: upper}
Suppose that $\{M_1,\dots, M_t\}$ is a set of maximal subgroups of the symmetric group $S_n$ such that (a) $k(M_i)\geq \frac 12n$ for every $i$; (b) if $i\neq j$, then $M_i$ and $M_j$ are not $S_n$-conjugate. Then
\begin{equation}\label{e: pp}
t\leq \frac{n}{2}+\Delta(n)+O\left(\frac{\log n}{\log \log n}\right),
\end{equation}
where $\Delta(n)$ is the number of divisors of $n$.
\end{proposition}

The main point in the proof of Proposition~\ref{p: upper} is to deal with the family of almost simple primitive subgroups of $S_n$; see Theorem \ref{p: locally valid}. The key ingredient is \cite[Theorem 1.2]{garzonigill}, which determines the almost simple primitive subgroups $G$ of $S_n$ such that $k(G)\geq \frac 12n$.

We remark that a theorem of Liebeck and Shalev gives a general upper bound for the number of conjugacy classes of maximal subgroup of $S_n$ of the form $\frac n2+o(n)$ \cite{lisha}. This immediately gives an upper bound for $m_I(S_n)$ and, in light of the easy fact that $m(S_n)\geq n-1$, yields Corollary~\ref{c: main} provided $n$ is large enough (this was observed also in \cite{GarLuc}). We note that in Proposition \ref{p: upper} we do not use \cite{lisha}, but we use \cite{garzonigill}, which relies on upper bounds for the number of conjugacy classes of almost simple groups of Lie type by Fulman--Guralnick \cite{FG1}.

We remark, moreover, that although Proposition~\ref{p: upper} only states an upper bound for the number of maximal subgroups with at least $\frac12 n$ conjugacy classes, results in \S\ref{s: upper} outline specific families of maximal subgroups. In particular, the first two terms of \eqref{e: pp} correspond to the intransitive and imprimitive subgroups of $S_n$, respectively.

This is important because our original aim in this paper was to prove that $|m_I(G)-\frac n2| =O(\log n)$. We have managed this with the lower bound but not with the upper, precisely because $\Delta(n)-2$, which is the number of conjugacy classes of maximal imprimitive subgroups of $S_n$, is not $O(\log n)$. To achieve our original aim, it would be sufficient to establish that, in the following question, $t\leq \frac n2+O(\log n)$. We state the question in terms of properties of $S_n$ -- it is easy enough to recast it as a number-theoretic question concerning partitions, similar to Question \ref{q: large_partitions} above.\footnote{Yet another way to think of this question uses the terminology of the previous footnote. We are effectively asking the following: Let $t$ be a positive integer and let $M_1,\dots, M_t$ be maximal subgroups of $S_n$ that are either intransitive or imprimitive. If $\{M_1^*, \ldots, M_t^*\}$ is independent, then how large can $t$ be?
}
\begin{question}
\label{q: partition_imprimitive}
For a positive integer $n$, how large can $t$ be such that we can find sets with the following properties?
\begin{enumerate}
    \item  $\{\calC_1,\dots, \calC_t\}$ is a set of conjugacy classes of $S_n$;
    \item $\{M_1,\dots, M_t\}$ is a set of maximal subgroups of $S_n$, all of which are intransitive or imprimitive;
    \item For $i=1,\dots, t$, $\calC_i\cap M_i=\emptyset$;
    \item For $i,j=1,\dots, t$, if $i\neq j$, then $\calC_i\cap M_j\neq \emptyset$.
\end{enumerate}
\end{question}

Proposition \ref{p: partitions} shows that $t>n/2-\log n$. In truth, we believe that, at least for large enough $n$, a MIG-set of $S_n$ of size $m_I(S_n)$ should involve only intransitive subgroups (in the sense that the set $J$ from Lemma \ref{l: basic_obs_mI} should contain only intransitive subgroups). This would imply that $m_I(S_n)\leq \frac n2$, and the problem of determining $m_I(S_n)$ would be reduced to the purely combinatorial problem addressed in Proposition \ref{p: partitions} and Question \ref{q: large_partitions}.

\subsection{Context}
\label{subsec: context_mIm}
The concept of invariable generation was introduced by Dixon, with the motivation of recognizing $S_n$ as the Galois group of polynomials with integer coefficients \cite{Dix92}. See for instance Kantor--Lubotzky--Shalev \cite{KLS} for interesting results related to invariable generation of finite groups.

In \cite{GarLuc}, the invariant $m_I(G)$ was introduced and studied. This is the ``invariable'' version of the invariant $m(G)$, which is the largest possible cardinality of a minimal generating set of $G$. 
See Lucchini \cite{min, min2} for results concerning $m(G)$ where $G$ is a general finite group.

Assume now $G=S_n$. It is easy to see that $m(S_n)\geq n-1$, by considering the $n-1$ transpositions $(1,2), \ldots, (n-1,n)$. Using CFSG, Whiston \cite{whi} proved that in fact $m(S_n)=n-1$. But more is true. Cameron--Cara \cite{Caca} showed that a minimal generating set of $S_n$ of size $n-1$ is very restrictive: either it is made of $n-1$ transpositions, or it is made of a transposition, some $3$-cycles, and some double transpositions (see \cite[Theorem 2.1]{Caca} for a precise statement).

One can hardly hope for a similar ``elegant'' result for $m_I(S_n)$, for the simple reason that a minimal invariable generating set of $S_n$ of size $t$ must contain $t$ distinct partitions which do not have a common partial sum. Still, it is true that in the proof of the lower bound in Theorem \ref{t: main}, we feel somewhat restricted about the choice of the relevant partitions -- but we are not able to make any precise statement in this direction.

In \cite{GarLuc}, it was shown that, if $G$ is a finite soluble group, then $m_I(G)=m(G)$, which in turn is equal to the number of complemented chief factors in a chief series of $G$. Moreover, it was asked whether $m_I(G)\leq m(G)$ is true for every finite group. It seems that, ``often'', for a finite non-soluble group $G$, the strict inequality $m_I(G)<m(G)$ holds. Corollary \ref{c: main} confirms this feeling in case $G=S_n$.

We recall, however, that $m_I(\PSL_2(p))=m(\PSL_2(p))$ for infinitely many primes $p$ (see \cite[Section 5]{GarLuc}).

\subsection{Structure of the paper and notation} In \S\ref{s: lower} we prove the lower bound on $m_I(G)$ given in Theorem~\ref{t: main}. In \S\ref{s: upper} we prove the upper bound on $m_I(G)$ given in Theorem~\ref{t: main}, along with Corollary~\ref{c: main}.

We will use exponential notation for partitions, so the partition $(a_1^{n_1}, a_2^{n_2}, \ldots, a_t^{n_t})$ has $n_1$ parts of length $a_1$, $n_2$ parts of length $a_2, \ldots$, and $n_t$ parts of length $a_t$. For a positive real number $x$, $\log(x)$ denotes a logarithm in base $2$. For a positive integer $x$, $\Delta(x)$ denotes the number of divisors of $x$.

\section{The lower bound}\label{s: lower}

In this section we prove the lower bound in Theorem \ref{t: main}. We first prove a lemma, then we prove Proposition ~\ref{p: partitions}, and finally we give a proof for the lower bound.


\begin{lemma}
\label{cruc}
Let $n$ and $i$ be positive integers, with $i < n/3$. Then there exist a partition $\mathfrak p_{i,n}$ of $n$ with the following properties:
\begin{enumerate}
    \item If $n \neq 4i+2$ and $(n,i) \neq (8,1)$, then $\mathfrak p_{i,n}$ does not have $i$ and $n-i$ as partial sums, and everything else is a partial sum.
    \item If $n=4i+2$ or $(n,i)=(8,1)$, then $\mathfrak p_{i,n}$ does not have $i, n-i$ and $\frac{n}{2}$ as partial sums, and everything else is a partial sum.
\end{enumerate}
\end{lemma}

\begin{proof}
Define
\[
\mathfrak p_{i,n} = (1^{i-1}, i+1, (i+2)^j, (i+1)^k, c)
\]
where $j \in \{0,1\}$, $k \geqslant 0$, and $i+1 \leqslant c \leqslant 2i+1$. To complete the definition we must specify $j,k$ and $c$. To do this we consider a partial sum $q$, adding from left to right: we first sum the $1$'s and $(i+1)$ to obtain $q=2i$. Now there are three cases:

\begin{enumerate}
 \item If $n-q\leqslant 2i+1$, then we set $c=n-q$.
 \item If $n-q=2i+2$, then we set $k=1$, $j=0$ and $c=i+1$.
 \item If $n-q\geqslant 2i+3$, then we set $j=1$, $k=0$ and set $q=2i+(i+2)=3i+2$.
\end{enumerate}

In the first and second cases, we are done; notice that the partition has the stated properties (in the first case we use the fact that $i<n/3$ to obtain $i+1\leqslant c\leqslant 2i+1$ as required). If we are in the third case, then we proceed in a loop as follows:

\begin{enumerate}
 \item If $n-q\leqslant 2i+1$, then we set $c=n-q$.
 \item If $n-q\geqslant 2i+2$, then we set $k=k+1$ and set $q=q+(i+1)$.
\end{enumerate}

It turns out that there is one situation -- when $n=4i+4$ and $(n,i) \neq (8,1)$ -- where our definition needs to be adjusted. In this case, we make the following definition:
\[
\mathfrak p_{i,n}=(1^{i-1}, i+1, i+3, i+1).
\]
Now our definition is complete. We now let $m$ be an integer such that $1 \leqslant m \leqslant n/2$ and we study when $m$ is a partial sum of $\mathfrak{p}_{i,n}$ with a view to proving items (1) and (2) of the lemma. 

Both items are clear for $m \leqslant 2i$, 
thus we may assume that $m \geqslant 2i+1$. In particular this means that $n\geq 4i+2$.

If $n=4i+2$, then we are in item (2) of the lemma, $\mathfrak p_{i,n}=(1^{i-1}, (i+1)^3)$, and the statement holds. 

If $n\geqslant 4i+3$, note first that $j=1$. Suppose, first, that $k=0$. There are two possibilities: first, if $c\neq i+2$, then $\mathfrak p_{i,n} = (1^{i-1}, i+1, i+2, c)$, and the statement holds. If instead $c=i+2$, then $n=4i+4$. If $(n,i)=(8,1)$, then $\mathfrak p_{1,8}=(2,3,3)$ and we are in item (2) of the lemma. Otherwise, we are in the exceptional case in our definition where $\mathfrak p_{i,n}=(1^{i-1}, i+1, i+3, i+1)$, and the result holds.

We are left with the case in which $k\geqslant1$, i.e. $\mathfrak p_{i,n}$ contains at least two $(i+1)$'s (excluding $c$ which may also equal $i+1$). 

We work here by induction on $m$: assuming that some $2i \leqslant m < n/2$ can be written \textit{without using c}, we want to show that the same holds for $m+1$. Since $m \geqslant 2i \geqslant i+1$, in writing $m$ without $c$ we have certainly used at least one of $i+1$ and $i+2$. We now divide into three cases.

\begin{enumerate}
\item In writing $m$ we have not used all $1$'s. Then add a $1$.
\item In writing $m$ we have not used $i+2$. Then remove an $i+1$ and add $i+2$.
\item In writing $m$ we have used all $1$'s and $i+2$. Suppose, first, that at least two $(i+1)$'s have not been used; then remove all $1$'s, remove $i+2$ and add two $(i+1)$'s and we are done. On the other hand, suppose (for a contradiction) that in writing $m$ as a partial sum all but one of the $i+1$'s have been used. Then $c+(i+1)> n/2$ and, since $c\leqslant 2i+1$, we obtain that $n/2< 3i+2$. However the partial sum $m$ has used all $1$'s, one $i+1$ and one $i+2$, so $m\geq 3i+2$. Since $m<n/2$, we get $n/2>3i+2$, which is a contradiction.\qedhere
\end{enumerate}
\end{proof}

\subsection{Proof of Proposition \ref{p: partitions}} Now we prove Proposition \ref{p: partitions}. The proof we give below is constructive -- we define an explicit set $X$ with the given properties. 
We have decided not to define the set $X$ outside of this proof, as the construction is built up in pieces as the proof proceeds. 

In deducing the lower bound in Theorem \ref{t: main}, we will be interested in the properties of the partitions of $X$ listed in the statement of Proposition \ref{p: partitions}, rather than their explicit construction. The paragraphs involving exceptions to this are labelled {\bf (C1)}, {\bf (C2)}, {\bf (C3)} and {\bf (C4)} in the following proof.

\begin{proof}[Proof of Proposition \ref{p: partitions}]
Throughout the proof, we will use the notation $\mathfrak p_{t,n}$ to refer to the partitions in the statement of Lemma \ref{cruc} (so we allow any partition having the properties of the statement).

If $5\leq n\leq 10$, we have $n/2-\log n <2$, and the statement is easy to check. Therefore assume $n\geq 11$.

{\bf (C1)} For $n=11$, we set $x_1=(2^2,3,4)$, $x_2=(1,3^2,4)$, $x_3=(1^2,9)$. For $n=12$, we set $x_1=(2^2,3,5)$, $x_2=(1,3,4^2)$, $x_3=(1^2,10)$.  The statement holds by setting $X=\{x_1, x_2, x_3\}$.



From now on we assume that $n > 12$. This has the advantage that in the proof that follows, all partitions of the form $\mathfrak p_{i,\ell}$ that we consider will have $\ell>8$ or $i >1$, and so we need not worry about the case $(\ell,i)=(8,1)$ mentioned in Lemma~\ref{cruc}.

For $1 \leqslant t < n/3$, define
\[
x_t=\mathfrak p_{t,n}.
\]
{\bf (C2)} We want to modify partition $x_1$. Namely, define
\[
x_1=
\begin{cases}
(3,5,2^k,4^j) & \text{if $n=14,16$}\\
(3,2^k,4^j) & \text{if $n=13,15,17$}\\
(3^\ell,7,2^k,4^j) & \text{if $n\geq 18$} 
\end{cases}
\]
where $\ell=1$ or $2$ according to whether $n$ is even or odd, and $j\in \{0,1\}$ is defined by the condition that $x_1$ has an odd number of cycles of even length (and $k$ is consequently uniquely defined). It is easy to see that, in every case, every $2 \leq i \leq n/2$ is a partial sum in $x_1$.


Now, in order to go further, we will use a slightly different method. The partitions we are going to define will depend on a parameter $j$. We could define all of them at once, but to give an idea of the overall strategy, let us go through the first step explicitly. 

We set $\alpha_1= \ceil{n/6-1}$ and, for every integer $\ceil{n/3} \leqslant t \leqslant t_1=5n/12$ we define
\[
x_{t} = (\mathfrak p_{\alpha_1, \alpha_1+t}, c_{t}).
\]
where $c_{t} = n-\alpha_1-t$. Let us justify this definition:
\begin{enumerate}
 \item[(a)] Observe first that $1 \leq \alpha_1<(\alpha_1+t)/3$ hence the partition $\mathfrak p_{\alpha_1, \alpha_1+t}$ is well-defined.
 \item[(b)] Next note that
\begin{align*}
  &c_{t}=n-\alpha_1-t \\
  &\quad > n-\left\lceil{\frac{n}{6}-1}\right\rceil-\frac{5n}{12} \\
  &\quad\quad\quad >\frac{5n}{12}\geq t > \alpha_1,
 \end{align*}
and so $\alpha_1$ and $t$ are not partial sums of $x_{t}$.
\item[(c)] We can easily check that either $4\alpha_1+2 > \alpha_1+t$ or else $(\alpha_1, t, n) = (1,5,12)$. The second possibility is excluded by our assumption $n > 12$. The first possibility implies that Lemma~\ref{cruc}(1) holds, and so all numbers up to $\alpha_1+t$ are partial sums, apart from $\alpha_1$ and $t$.  
\item[(d)] Finally observe that
\begin{align}\label{n2}
 \alpha_1+t &\geq \left\lceil{\frac{n}{6}-1}\right\rceil + \left\lceil{\frac{n}{3}}\right\rceil \\
 &\geqslant \left\lceil{\frac{n}{2}-1}\right\rceil.\nonumber
\end{align}
We conclude that all numbers up to $n$ are partial sums in $x_t$ apart from $\alpha_1$, $t$ and (possibly) $n/2$. In fact, checking \eqref{n2} more carefully, it is clear that $\alpha_1+t\geq \lfloor n/2 \rfloor$ unless $n\equiv 0\pmod 6$ and $t=n/3$.
\end{enumerate}

{\bf Conclusion~1}: For $n/3<t\leq 5n/12$, the partition $x_t$ admits all partial sums up to $n$ except $\alpha_1$ and $t$.

{\bf Conclusion~2}: For $t=n/3$, the partition $x_t$ admits all partial sums up to $n$ except $\alpha_1$ and $t$ and (if $n$ is even) $n/2$.

Now our aim is to extend this definition to other parameters $t_i$ that are larger than $t_1$. More precisely, for integer $1 \leqslant j \leqslant \log(n/6)$ we define
\[
\alpha_j = \left\lceil{\frac{n}{2^{j-1}\cdot 6} -1}\right\rceil, \hspace{20pt}
t_j = \frac{(2^{j-1} \cdot 6 -1)n}{2^j \cdot 6}.
\]
For $j=1$ this is consistent with the previous definition. Now for integers $2 \leqslant j \leqslant \log(n/6)$ and $\floor{t_{j-1}} +1 \leqslant t \leqslant \floor{t_j}$ we define  
\[
x_t = (\mathfrak p_{\alpha_j, \alpha_j+t}, c_t)
\]
where $c_t = n-\alpha_j-t$. Now, similarly to before, we must check four properties. Assume that $j\geqslant 2$.
\begin{enumerate}
\item[(a)] Observe that $1 \leq \alpha_j <  (\alpha_j + \floor{t_{j-1}} +1) / 3$ and so the partition $\mathfrak p_{\alpha_j, \alpha_j+t}$ is well-defined.
\item[(b)] Next note that $c_t=n-\alpha_j - t > t > \alpha_j$ and so $\alpha_j$ and $t$ are not partial sums of $x_t$.
\item[(c)] Notice that the second case of Lemma \ref{cruc} does not occur. Indeed, $4\alpha_j+2$ is strictly smaller than $\alpha_j +\floor{t_{j-1}}+1$. Moreover, it is easy to check that the case $(\ell,i)=(8,1)$ cannot occur. 
\item[(d)] Finally observe that $\alpha_j+\floor{t_{j-1}}+1 \geq \floor{n/2}$, and we conclude that all numbers up to $n$ are partial sums in $x_t$ apart from $\alpha_1$ and $t$.
\end{enumerate}
{\bf Conclusion~3}: Set $m=\floor{\log(n/6)}$. For $j=2,\dots, m$ and $\floor{t_{j-1}} +1 \leqslant t \leqslant \floor{t_j}$, the partition $x_t$ admits all partial sums up to $n$, except $\alpha_j$ and $t$.

We have now constructed $\floor{t_m}$ partitions of $n$; set $X_0=\{x_1, \ldots, x_{\floor{t_m}}\}$. Notice that, by the choice of $m$, $2^{m+1}\cdot 6 > n$. Then 
\begin{align}\label{e: tm}
|X_0|=\floor{t_m} &> \frac{(2^{m-1} \cdot 6 -1)n}{2^m \cdot 6}-1 \\
&= \frac{n}{2} - \frac{n}{2^m \cdot 6} - 1 \nonumber \\
&> \frac n2 -3.\nonumber
\end{align}
Now we will remove some elements from $X_0$. First, observe that $\alpha_j<n/6$ for every $j$; we start by taking the subset $X$ obtained by removing $x_{\alpha_j}$ for every $j \geqslant 1$. 

{\bf (C3)} Lemma~\ref{cruc}, and the three conclusions listed above imply that, for each $t$ satisfying $1\leq t\leq t_m$, the partition $x_t$ is the unique partition in $X$ which does not admit $t$ as a partial sum. Now we divide into two cases.
\begin{enumerate}
\item There exists an integer belonging to the interval $(t_m, n/2]$ which is a partial sum for all $x_t\in X$. Then, the minimum such integer is $\floor{t_m}+1$; we add one further partition to $X$:
\[
z = (1^{\floor{t_m}}, n-\floor{t_m}).
\]
\item No integer in $(t_m, n/2]$ is a partial sum in all $x_t$'s. Observing that $t_m\leq n/2-1$, Lemma~\ref{cruc} and the three conclusions above imply that $t_m = n/2-1$, i.e., $n=2^m \cdot 6$.  In this case we could leave $X$ unchanged, and the statement would be proved. However, we prefer to immediately modify the set $X$. 
Notice that $\alpha_m=1$ and $t_{m-1}=n/2-2$; it follows that $x_1=\mathfrak p_{1, n} \notin X$, and $X$ contains a unique partition, namely $x_{t_m}$, of the form $(\mathfrak p_{1, 1+t}, c_t)$. Then, we remove such partition and we reintegrate the partition $x_1$ in $X$. Moreover, we add to $X$ one further partition 
\[
z = (1^{n/2-2}, n/2+2).
\]
\end{enumerate}
{\bf (C4)} Our construction is finished. Let us make one observation, before concluding the proof. In case (2) above, by construction $x_1 \in X$. We claim that the same holds in case (1). Indeed, one can easily check that $\alpha_m = 1$ if and only if $n=2^m \cdot 6$, and otherwise $\alpha_j > 1$ for every $j$. Therefore, in case (1), in our procedure we did not remove $x_1$ from $X_0$, hence clearly $x_1 \in X$.

We are now ready to conclude the proof of the statement. The considerations above imply that items (1) and (2) of the statement hold. 
Regarding item (3),
\begin{align*}
 |X| &\geq |X_0|+1 - \log(n/6) \\
 &\geq \frac{n}{2}-2-\log n + \log 6 \\
 &> \frac n2 - \log n.
\end{align*}
The proposition is now proved.
\end{proof}

We now deduce the lower bound of Theorem~\ref{t: main} from Proposition \ref{p: partitions}.

\begin{proof}[Proof of the lower bound of Theorem~\ref{t: main}]
 For $n=2$, $n/2-\log n=0$ and the statement is trivial. For $3\leq n\leq 10$, we have $n/2-\log n<2$. Since certainly $m_I(S_n) \geq 2$, the statement holds and we may assume $n \geq 11$.

Consider the set $X$ of partitions constructed in the proof of Proposition \ref{p: partitions}. We will consider the elements of $X$ as conjugacy classes of $S_n$. We want to show that $X$ is a MIG-set for $S_n$.

It is easy to check the statement for $n=11,12$ (see the paragraph {\bf (C1)} in the proof of Proposition \ref{p: partitions}). Assume now $n \geq 13$. 
By Proposition \ref{p: partitions}(1), the classes of $X$ cannot have non-empty intersection with an intransitive subgroup of $S_n$. On the other hand, by Proposition \ref{p: partitions}(2), if we drop one class from $X$, then the remaining classes have non-empty intersection with some intransitive subgroup.

Now we deal with transitive groups. Note that $X$ is not contained in $A_n$, since $x_1$ corresponds to an odd permutation (see the paragraphs {\bf (C2)} and {\bf (C4)}). Moreover, a power of $x_1$ corresponds to a cycle of prime length fixing at least $3$ points, which belongs to no primitive group different from $A_n$ and $S_n$ by a classical theorem of Jordan. Assume now the classes of $X$ preserve a partition of $\{1, \ldots, n\}$ made of $r>1$ blocks of size $k>1$. Recall that $X$ contains a partition $z=(1^{n-\ell}, \ell)$, with $\ell=n/2+2$ or $\ell=n-\floor{t_m}$ (see the paragraph {\bf (C3)}). By eq. \eqref{e: tm} in the proof of Proposition \ref{p: partitions}, we have $n/2-3 < \floor{t_m}< n/2$, and in particular $n/2<\ell < n/2+3$. We have that $k$ must divide $\ell$. Since $k$ also divides $n$, we get $k < 6$. If $n\neq 15$, then $x_1$ cannot preserve blocks of size at most $5$. If $n=15$, we note that $X$ contains $x_4=\mathfrak p_{4,15}$, and we may take $\mathfrak p_{4,15}=(1^3,5,7)$, which does not preserve any nontrivial partition of $\{1,\dots, 15\}$. The proof is now concluded.
\end{proof}

\section{The upper bound}
\label{s: upper}

In this section we prove the upper bound in Theorem \ref{t: main}, and we prove Corollary \ref{c: main}. Our main tool is the following result, which follows quickly from \cite{garzonigill}. Recall that $k(G)$ denotes the number of conjugacy classes of a finite group $G$.

\begin{theorem}\label{p: locally valid}
Let $G$ be a maximal almost simple primitive subgroup of $S_n$, and assume $k(G)\geq \frac n2$. Then one of the following occurs:
\begin{enumerate}
    \item $G$ is listed in Table~\ref{tab: max};
    \item $G=A_n$, or $G=S_d$ and the action of $G$ on $n$ points is isomorphic to the action on the set of $k$-subsets of $\{1, \ldots, d\}$ for some $2\leq k<d/2$.
    \item $G=\PGammaL_d(q)$ and the action of $G$ on $n$ points is isomorphic to the action on the set of $1$-subspaces of $\F_q^d$.
\end{enumerate}
\end{theorem}

Note that the subgroups mentioned at item (2) satisfy $n=\binom{d}{k}$ for some integer $k$ with $1\leq k<d/2$; and the subgroups mentioned at item (3) satisfy $n=(q^d-1)/(q-1)$.

\begin{table}
\centering
\begin{tabular}{ccc}
 \hline\noalign{\smallskip}
$n$ & $G$ & $k(G)$  \\
\noalign{\smallskip}\hline\noalign{\smallskip}
22 & $M_{22}.2$ &  21\\
40 & $\SU_4(2).2$ & 25 \\
45 & $\SU_4(2).2$ & 25 \\
\noalign{\smallskip}\hline\noalign{\medskip}
\end{tabular}
\caption{Some maximal almost simple primitive subgroups, $G$ of $S_n$, for which $k(G)\geq \frac{n}{2}$. In every case there is a single $S_n$-conjugacy class of primitive subgroups isomorphic to $G$.
}\label{tab: max}
\end{table}

\begin{proof}
The statement follows from \cite[Theorem 1.2]{garzonigill}, by checking with \cite{GAP} which of the entries in \cite[Table 1]{garzonigill} correspond to maximal subgroups of $S_n$.
\end{proof}

As we observed in the introduction, the upper bound in Theorem \ref{t: main} follows immediately from Proposition~\ref{p: upper}, which we prove now.

\begin{proof}[Proof of Proposition~\ref{p: upper}]

We make use of the families of maximal subgroups given in the Aschbacher--O'Nan--Scott theorem, in particular the description given in \cite {LPS}. 

\begin{enumerate}
    \item {\bf Intransitive subgroups}: There are exactly $\lfloor\frac{n}{2}\rfloor$ conjugacy classes of these.
    \item {\bf Imprimitive subgroups}: There are $\Delta(n)-2$ of these, where $\Delta$ is the divisor function.
    \item {\bf Affine subgroups}: There is at most $1$ conjugacy class of these.
    \item {\bf Almost simple subgroups}: If $M$ is almost simple and $k(M)\geq \frac n2$, then it is among the possibilities listed by Theorem~\ref{p: locally valid}, as follows.
\begin{enumerate}
    \item There are three possibilities for degrees $22, 40, 45$ listed in Table~\ref{tab: max}.
    \item There is at most one conjugacy class of maximal subgroups isomorphic to $\PGammaL_d(q)$ whenever $n=\frac{q^d-1}{q-1}$; we let $a_n$ be the number of pairs $(q,d)$ where $q$ is a prime power, $d$ is a positive integer, and $\frac{q^d-1}{q-1}=n$.
    \item There is at most one conjugacy class of maximal subgroups with socle $A_d$ whenever $n=\binom{d}{k}$ for some $k$; we let $b_n$ be the number of pairs $(d,k)$ where $d$ and $k$ are positive integers with $k\leq d/2$ and $\binom{d}{k}=n$. 
    \end{enumerate}
    \item {\bf Diagonal subgroups}: \cite[Theorem 1.1]{garzonigill} states that $k(M)<\frac n2$ in this case, so we can ignore these subgroups. 
    \item {\bf Product action subgroups}: In this case we have maximal subgroups isomorphic to $S_d\wr S_k$, where $n=d^k$ and $k>1$.  For fixed values of $d$ and $k$, there is one conjugacy class, thus the number of conjugacy classes in $S_n$ is equal to the number of pairs $(d,k)$ where $d$ and $k$ are positive integers with $k>1$ and $n=d^k$; we write this number as $c_n$.
    \item {\bf Twisted wreath subgroups}: These are never maximal, as they are defined to be subgroups of groups with a product action \cite{LPS} and so can be ignored (and in any case, $k(M)<\frac n2$ by \cite[Theorem 1.1]{garzonigill}).
\end{enumerate}

Observe that the number of conjugacy classes of maximal subgroup in $S_n$ that are either imprimitive, affine, or given in Table \ref{tab: max} is at most $\Delta(n)-1$. Therefore, if $\{M_1, \ldots, M_t\}$ is a set of maximal subgroups as in the statement, we have
\begin{equation}
\label{eq: upper bound}
t\leq \floor{\frac n2} + \Delta(n) + a_n + b_n + c_n -1.
\end{equation}
(We will use this in the proof of Corollary \ref{c: main}.) In order to prove Proposition~\ref{p: upper}, it is clearly enough to show that 
\[a_n+b_n+c_n =O\left(\frac{\log n}{\log \log n}\right).\]

To bound $a_n$, observe that if
    \[
    \frac{q_1^{d_1}-1}{q_1-1} = \frac{q_2^{d_2}-1}{q_2-1},
    \]
then $q_1$ and $q_2$ must be coprime. We obtain that $a_n$ must be bounded above by the number of distinct prime divisors of $n-1$. 
In \cite{robin} it is proved that this number is
\[
O\left(\frac{\log(n-1)}{\log \log(n-1)}\right),
\]
whence the same upper bound holds for $a_n$.

To bound $b_n$ we refer to a result of Kane \cite{kane2}, which asserts that\footnote{Singmaster's conjecture \cite{singmaster} asserts that $b_n$ is bounded above by an absolute constant; de Weger proposes that in fact this constant can be taken to be $4$ \cite{dew}, and evidence for the veracity of this conjecture is given in \cite{bbd}; in particular this is known to be true if $n\leq 10^{60}$.}
\[
b_n = O\left(\frac{\log n \log \log \log n}{(\log\log n)^3}\right).
\]

To bound $c_n$, we first recall (see \cite[Theorem 13.12]{Apo})
that, for a positive integer $x$, 
\[
\Delta(x)\leq \text{exp}_2\left\{\frac{(1+o(1))\log x}{\log \log x}\right\}.
\]
Now consider the prime factorization of $n$: $n=p_1^{a_1}\cdots p_t^{a_t}$. 
If $n=d^k$ then $p_1\cdots p_t$ divides $d$ and $k$ divides $a:=\text{gcd}\{a_1, \ldots, a_t\}$. Therefore, the number of choices for $k$ is at most the number of divisors of $a$ different from $1$. Now note that $a\leq \log n$, 
and therefore
\[
c_n \leq \text{exp}_2\left\{\frac{(1+o(1))\log \log n}{\log \log \log n}\right\}.
\]

In particular we see that each of $a_n$, $b_n$, $c_n$ is $O(\log n/\log \log n)$. This proves the proposition.
\end{proof}

We conclude with the proof of Corollary~\ref{c: main}.

\begin{proof}[Proof of Corollary~\ref{c: main}]
It is easy to see that $\{(1,2), (2,3), (3,4), \dots, (n-1, n)\}$ is a minimal generating set of size $n-1$. It is, therefore, enough to show that $m_I(S_n)< n-1$.

From eq. \eqref{eq: upper bound} in the proof of Proposition \ref{p: upper}, we deduce that
\[
m_I(S_n)\leq \floor{\frac n2} + \Delta(n) + a_n + b_n + c_n -1,
\]
therefore it is sufficient to show that $\Delta(n)+a_n+b_n+c_n<n/2$. 
Very weak estimates are enough here. First assume that $n\geq 71$.

As remarked in the proof of Proposition \ref{p: upper}, $a_n$ is bounded above by the number of distinct prime divisors of $n-1$, which is at most $\log n$. Moreover, $c_n$ is bounded above by $\max \{\Delta(x) : x\leq \floor{\log n}\}$, which is at most $\log n$. Let us consider $b_n$. Let $(d_1,k_1), \dots, (d_b, k_b)$ be pairs such that $\binom{d_i}{k_i}=n$ for all $i=1,\dots, b$. Order so that $i<j$ implies that $k_i< k_j$ and observe that then $k_b \geq b$ and $d_b\geq 2b$. This implies that $n\geq \binom{2b}{b}> 2^b$. In particular $b<\log n$.

Finally we need to bound $\Delta(n)$. For every real number $a\in (0,n]$, we have $\Delta(n)< n/a +a$. By choosing $a=\sqrt n$, we deduce $\Delta(n)< 2\sqrt n$.

Therefore $\Delta(n)+a_n+b_n+c_n<2\sqrt n +3\log n$, and it is sufficient to show that $2\sqrt n +3\log n\leq n/2$. Since $n\geq 71$, this is indeed the case.

For $n\leq 70$ we use \cite{GAP} to find that, except when $n\in\{5,6,8,12\}$, $S_n$ has less than $n-1$ conjugacy classes of maximal subgroup and the result follows immediately. 

For the remaining cases, say that two cycle types are \textit{equivalent} if one is a power of the other one; e.g., $(2,2)$ is equivalent to $(4)$, $(2,3,3)$ is equivalent to $(2,1^6)$, etc. Note that a MIG-set of $S_n$ of size $t$ must contain $t$ pairwise non-equivalent cycle types.

For $n\in \{5,8,12\}$, there are exactly $n-1$ conjugacy classes of maximal subgroups of $S_n$. However, in each case, there is one which does not intersect at least $n-3$ pairwise non-equivalent cycle types, and the result follows. (For $n=5$ we may take $\AGL_1(5)$, for $n=8$ we may take $\PGL_2(7)$, and for $n=12$ we may take $\PGL_2(11)$.)

For $n=6$, we note that a MIG-set of size $t\geq 5$ must contain $5$ distinct non-trivial cycle types, each of which intersects non-trivially $4$ pairwise non-conjugate maximal subgroups. A direct check shows that the cycle types with this property are
\[
(2), (2^2), (2^3), (3), (3^2), (4), (4,2). \tag*{$(\star)$}
\]
Now, a set of $4$ pairwise non-equivalent cycle types, each intersecting non-trivially $\PGL_2(5)$, must contain the cycle type $(5)$, which does not appear in $(\star)$. Therefore, we deduce that a MIG-set of size $t\geq5$ must contain $5$ distinct cycle types, each of which intersects non-trivially $4$ pairwise non-conjugate maximal subgroups, not isomorphic to $\PGL_2(5)$. We see that $(4)$ does not have this property. All other cycle types appearing in $(\star)$ intersect non-trivially $S_3\wr S_2$, and we deduce that $m_I(S_6)<5$, as wanted.
\end{proof}

\bibliography{references}

\begin{thebibliography}{{Whi}00}

\bibitem[Apo76]{Apo}
T.~M. Apostol.
\newblock {\em Introduction to analytic number theory}.
\newblock Springer-Verlag, New York-Heidelberg, 1976.
\newblock Undergraduate Texts in Mathematics.

\bibitem[BBD17]{bbd}
A.~{Blokhuis}, A.~{Brouwer}, and B.~{De Weger}.
\newblock {Binomial collisions and near collisions.}
\newblock {\em {Integers}}, 17:paper a64, 8, 2017.

\bibitem[CC02]{Caca}
P.~J. Cameron and P.~Cara.
\newblock Independent generating sets and geometries for symmetric groups.
\newblock {\em Journal of Algebra}, 258(2):641--650, 2002.

\bibitem[Dix92]{Dix92}
J.~D. Dixon.
\newblock Random sets which invariably generate the symmetric group.
\newblock {\em Discrete Mathematics}, 105(1-3):25--39, 1992.

\bibitem[dW97]{dew}
B.~M.~M. de~{Weger}.
\newblock {Equal binomial coefficients: Some elementary considerations.}
\newblock {\em {J. Number Theory}}, 63(2):373--386, 1997.

\bibitem[FG12]{FG1}
J.~{Fulman} and R.~M. {Guralnick}.
\newblock {Bounds on the number and sizes of conjugacy classes in finite
  Chevalley groups with applications to derangements.}
\newblock {\em {Trans. Am. Math. Soc.}}, 364(6):3023--3070, 2012.

\bibitem[GAP19]{GAP}
The GAP~Group.
\newblock {\em {GAP -- Groups, Algorithms, and Programming, Version 4.10.2}},
  2019.

\bibitem[GG20]{garzonigill}
D.~Garzoni and N.~Gill.
\newblock On the number of conjugacy classes of a primitive permutation group
  with nonabelian socle.
\newblock {\em arXiv preprint arXiv:2012.05547}, 2020.

\bibitem[GL20]{GarLuc}
D.~Garzoni and A.~Lucchini.
\newblock Minimal invariable generating sets.
\newblock {\em Journal of Pure and Applied Algebra}, 224(1):218--238, 2020.

\bibitem[{Kan}07]{kane2}
D.~M. {Kane}.
\newblock {Improved bounds on the number of ways of expressing \(t\) as a
  binomial coefficient.}
\newblock {\em {Integers}}, 7(1):paper a53, 7, 2007.

\bibitem[KLS11]{KLS}
W.~M. Kantor, A.~Lubotzky, and A.~Shalev.
\newblock Invariable generation and the {C}hebotarev invariant of a finite
  group.
\newblock {\em J. Algebra}, 348:302--314, 2011.

\bibitem[LPS88]{LPS}
M.~W. {Liebeck}, C.~E. {Praeger}, and J.~{Saxl}.
\newblock {On the O'Nan-Scott theorem for finite primitive permutation groups.}
\newblock {\em {J. Aust. Math. Soc., Ser. A}}, 44(3):389--396, 1988.

\bibitem[LS96]{lisha}
M.~W. {Liebeck} and A.~{Shalev}.
\newblock {Maximal subgroups of symmetric groups.}
\newblock {\em {J. Comb. Theory, Ser. A}}, 75(2):341--352, 1996.

\bibitem[Luc13a]{min}
A.~Lucchini.
\newblock The largest size of a minimal generating set of a finite group.
\newblock {\em Archiv der Mathematik}, 101(1):1--8, 2013.

\bibitem[Luc13b]{min2}
A.~Lucchini.
\newblock Minimal generating sets of maximal size in finite monolithic groups.
\newblock {\em Archiv der Mathematik}, 101(5):401--410, 2013.

\bibitem[Rob83]{robin}
G.~Robin.
\newblock Estimation de la fonction de {T}chebychef $\theta$ sur le
  $k$-i{\`e}me nombre premier et grandes valeurs de la fonction $\omega(n)$
  nombre de diviseurs premiers de $n$.
\newblock {\em Acta Arithmetica}, 42(4):367--389, 1983.

\bibitem[{Sin}71]{singmaster}
D.~{Singmaster}.
\newblock {How often does an integer occur as a binomial coefficient?}
\newblock {\em {Am. Math. Monthly}}, 78:385--386, 1971.

\bibitem[{Whi}00]{whi}
J.~{Whiston}.
\newblock {Maximal independent generating sets of the symmetric group.}
\newblock {\em {J. Algebra}}, 232(1):255--268, 2000.

\end{thebibliography}
\bibliographystyle{alpha}

\end{document}